\documentclass[12pt]{article}

\usepackage{graphicx}
\usepackage{verbatim}
\usepackage{amsmath}
\usepackage{amssymb}  % for \nmid
\usepackage{url}  % for \nmid
\usepackage{setspace}
\usepackage{comment}

\newcommand{\bits}[1]{\{0,1\}^{#1}}
\newcommand{\Z}{\mathbb{Z}}
\newcommand{\N}{\mathbb{N}}
\newcommand{\Q}{\mathbb{Q}}

\newcommand{\C}{\mathbb{C}}

\newcommand{\NP}{{\rm NP }}
\newcommand{\FROB}{{\rm FROB }}

\newcommand{\s}[1]{\s_{#1}}

\newcommand{\monus}{\;\raise.5ex\hbox{{${\buildrel
   \ldotp\over{\hbox to 6pt{\hrulefill}}}$}}\;}

\newcounter{savenumi}

\newtheorem{theoremfoo}{Theorem}[section] %by chapter in report style
\newenvironment{theorem}{\pagebreak[1]\begin{theoremfoo}}{\end{theoremfoo}}

\newtheorem{lemmafoo}[theoremfoo]{Lemma}
\newenvironment{lemma}{\pagebreak[1]\begin{lemmafoo}}{\end{lemmafoo}}
\newtheorem{conjecturefoo}[theoremfoo]{Conjecture}

\newtheorem{conventionfoo}[theoremfoo]{Convention}

\newtheorem{porismfoo}[theoremfoo]{Porism}

\newtheorem{gamefoo}[theoremfoo]{Game}

\newtheorem{corollaryfoo}[theoremfoo]{Corollary}

\newtheorem{openfoo}[theoremfoo]{Open Problem}

\newtheorem{exercisefoo}{Exercise}

\newcommand{\fig}[1] %usage:\fig{file}
{
\begin{figure}
\begin{center}
\input{#1}
\end{center}
\end{figure}
}

\newtheorem{potanafoo}[theoremfoo]{Potential Analogue}

\newtheorem{notefoo}[theoremfoo]{Note}

\newtheorem{notabenefoo}[theoremfoo]{Nota Bene}

\newtheorem{nttn}[theoremfoo]{Notation}
\newenvironment{notation}{\pagebreak[1]\begin{nttn}\rm}{\end{nttn}}

\newtheorem{empttn}[theoremfoo]{Empirical Note}

\newtheorem{examfoo}[theoremfoo]{Example}

\newtheorem{dfntn}[theoremfoo]{Def}
\newenvironment{definition}{\pagebreak[1]\begin{dfntn}\rm}{\end{dfntn}}

\newtheorem{propositionfoo}[theoremfoo]{Proposition}

\newenvironment{proof}
    {\pagebreak[1]{\narrower\noindent {\bf Proof:\quad\nopagebreak}}}{\QED}

\newcommand{\yyskip}{\penalty-50\vskip 5pt plus 3pt minus 2pt}
\newcommand{\blackslug}{\hbox{\hskip 1pt
       \vrule width 4pt height 8pt depth 1.5pt\hskip 1pt}}
\newcommand{\QED}{{\penalty10000\parindent 0pt\penalty10000
       \hskip 8 pt\nolinebreak\blackslug\hfill\lower 8.5pt\null}
       \par\yyskip\pagebreak[1]}

\newcommand{\BBB}{{\penalty10000\parindent 0pt\penalty10000
       \hskip 8 pt\nolinebreak\hbox{\ }\hfill\lower 8.5pt\null}
       \par\yyskip\pagebreak[1]}

\newtheorem{factfoo}[theoremfoo]{Fact}

\title{Hilbert's Tenth Problem: Refinements and Variants}
\author{William Gasarch\thanks{The University of Maryland at College Park, gasarch@umd.edu}}

\date{}

\begin{document}
\maketitle

\begin{abstract}
Hilbert's 10th problem, stated in modern terms, is

{\it Find an algorithm that will, given $p\in \Z[x_1,\ldots,x_n]$,
determine if there exists $a_1,\ldots,a_n \in \Z$ such that $p(a_1,\ldots,a_n)=0$.}

Davis, Putnam, Robinson, and Matijasevi\v{c}
showed that there is no such algorithm. 
We look at what happens (1) for fixed degree and number of variables,
(2) for particular equations, and (3) for variants which reduce
the number of variables needed for undecidability results.
\end{abstract}

\section{This Article's Origin}

This article is a long version of an article based on a blog.
What? I give the complete history. 

\begin{enumerate}
\item 
On May 4, 2020 I wrote a blog about Hilbert's 10th problem: 
This blog caught the attention of Thomas Erlebach who invited me to write a full article
for {\it The Bulletin of the European Association for Theoretical Computer Science (BEATCS)}
on this topic. 

\item 
The article: 
{\it Hilbert's Tenth Problem for Fixed $d$ and $n$} appeared in the 
{\it Bulletin of the European Association for Theoretical Computer Science (BEATCS)},
Vol 133, February 2021.

\item 
After it appeared I made a few updates to my copy of the article, added a few whole
new sections. This article is the result.

\end{enumerate}

\section{Hilbert's Tenth Problem}\label{se:h10}

In 1900 Hilbert proposed 23 problems for mathematicians to work on over
the next 100 years (or longer). The 10th problem, stated in modern terms, is

\begin{quote}
Find an algorithm that will, given $p\in \Z[x_1,\ldots,x_n]$,
determine if there exists $a_1,\ldots,a_n \in \Z$ such that $p(a_1,\ldots,a_n)=0$.
\end{quote}

Hilbert probably thought this would inspire much deep number theory, and it did
inspire some. But the work on this problem took a very different direction. 
Davis, Putnam, and Robinson~\cite{DavisPutnam} 
showed that determining if an exponential diophantine equation has a solution in 
$\Z$ is undecidable. Their proof coded Turing machines into such equations.
Matijasevi\v{c}~\cite{HilbertTenth} extended their work by showing how to 
replace the exponentials with polynomials. 
Hence the algorithm that Hilbert wanted is not possible. 
For a self contained proof from soup to nuts see Davis' exposition~\cite{Davis-1973}. 
For more about both the proof and the implications of the result see the book
of 
Matijasevi\v{c}~\cite{HilbertTenthbook}. 

The undecidability result 
raises the question of what happens for {\it particular}
numbers of variables $n$ and degree $d$.
I thought that 
surely there must be a grid on the web where the $d$-$n$-th entry is
\begin{itemize}
\item
D if the problem for degree $\le d$, and $\le n$ variables is Decidable.
\item
U if the problem for degree $\le d$, and $\le n$ variables is Undecidable.
\item
? if the status of the problem for degree $\le d$, and $\le n$ variables is unknown.
\end{itemize}

There is a graph in a paper written in German, by Bayer et al.~\cite{H10grid}, that has
the information I want, though it's hard to read and seems to only deal
with very large degrees. Putting that aside I ask
{\it Why has the quest for a grid not gotten more attention?}
Here are some speculations. 

\begin{enumerate}
\item
Logicians work on showing that determining if there is a solution in $\N$
is undecidable.
Number theorists worked on showing that determining if there is a solution in $\Z$ 
is decidable.
Since Logicians worked in $\N$ and Number Theorists in $\Z$,
a grid would need to reconcile these two related problems.

\item
There is a real dearth of positive results, so a grid would not be that interesting.

\item
The undecidable results often involve rather large values of $d$, so the grid would be
hard to draw.

\item
Timothy Chow offered this speculation in an email to me: 
{\it One reason there isn't already a website of the type you envision is
that from a number-theoretic (or decidability) point of view, parameterization by
degree and number of variables is not as natural as it might seem at first 
glance. The most fruitful lines of research have been geometric, and so
geometric concepts such as smoothness, dimension, and genus are more
natural than, say, degree. A nice survey by a number theorist is the book
\emph{Rational Points on Varieties} by Bjorn Poonen~\cite{Poonen-2017}. 
Much of it is highly technical; however, reading the preface is very 
enlightening. Roughly speaking, the current state of the art is that
there is really only one known way to prove that a system of
Diophantine equations has no rational solution.}
\end{enumerate}

Since the grid is hard to draw we do not present it in this paper. 
However, this paper does 
collect up all that is known and points to open problems.
None of the results are mine. 

\newcommand{\HtenZ}{{\rm H}\Z}
\newcommand{\HtenN}{{\rm H}\N}
\newcommand{\HtenQ}{{\rm H}\Q}
\newcommand{\D}{{\rm D}}
\newcommand{\U}{{\rm U}}

In Section~\ref{se:NZ} we will relate the problem of seeking solutions in 
$\Z$ with the problem of seeking solutions in $\N$. 
In Section~\ref{se:dn} we will look at what is known for fixed $d,n$
both for solutions over $\N$ and solutions in $\Z$. 
In Subsection~\ref{se:H10U} we will present values of $(d,n)$ where 
the problem is undecidable. 
In Subsection~\ref{se:H10D} we will present values of $(d,n)$ where 
the problem is decidable. 
In Subsection~\ref{se:32} we discuss why $\HtenZ(3,2)$ has almost been proved
decidable. 
In Section~\ref{se:part} we will discuss classes of polynomials with other
conditions added. 
In Section~\ref{se:disc} we will discuss the vast area between 
the decidable and undecidable cases. 
In Section~\ref{se:variants} we discuss variants of Hilbert's 10th problem
that lead to getting undecidability results with polynomials in
fewer variables.
In Section~\ref{se:disc2} we will briefly present 
Matijasevi\v{c}'s discussion of what Hilbert really wanted in contrast to
what happened. 

\section{Definitions and Reconciling $\N$ with $\Z$}\label{se:NZ}

\begin{notation}~
\begin{enumerate}
\item
$\HtenZ(d,n)$ is the problem where  the degree is $\le d$, the number of
variables is $\le n$, and we seek a solution in $\Z$.
\item
$\HtenN(d,n)$ is the problem where  the degree is $\le d$, the number of
variables is $\le n$, and we seek a solution in $\N$.
\item
$\HtenZ(d,n)=\D$ means that there is an algorithm to decide $\HtenZ(d,n)$.
\item
$\HtenZ(d,n)=\U$ means that there is no algorithm to decide $\HtenZ(d,n)$.
\item
Similarly for $\HtenN(d,n)$ equal to $\D$ or $\U$.
\end{enumerate}
\end{notation}

The four parts of the next lemma are usually stated with $x_1,x_2,x_3\in\N$ or
$x_1,x_2,x_3,x_4\in\N$, and not in the iff form we use. However, we need these
statements in the form we present them. 

\begin{lemma}\label{le:three}~
\begin{enumerate}
%\item
%$x\in\N$  iff there exists $x_1,x_2,x_3,x_4\in\Z$ such that
%$$x=x_1^2+x_2^2+x_3^2+x_4^2.$$
\item
$x\in\N$ and $x$ is not of the form $4^a(8b+7)$ (where $a,b\in\N$) iff  there exists
 $x_1,x_2,x_3 \in\Z$ such that
$$x=x_1^2+x_2^2+x_3^2.$$
\item
$x\in\N$ and $x\equiv 1 \pmod 4$ iff  there exists $x_1,x_2,x_3\in\Z$ such that 
(1) $x_1,x_2\equiv 0\pmod 2$, (2) $x_3\equiv 1 \pmod 2$, and
$$x=x_1^2+x_2^2+x_3^2.$$
\item 
$n\in\N$ iff there exists $x_1,x_2,x_3\in\Z$ such that $n=x_1^2+x_2^2+x_3^2+x_3$.
\end{enumerate}
\end{lemma}

\begin{proof}~

%\noindent
%1) This is Lagrange's four-square theorem.

%\smallskip

\noindent
1) This is Legendre's three-square theorem. It is sometimes called the
Gauss-Legendre Theorem.

\smallskip

\noindent
2) Since $x\equiv 1\pmod 4$, $x$ satisfies the hypothesis of Part 1. Hence
there exists $x_1,x_2,x_3$ such that

$$x=x_1^2+x_2^2+x_3^2.$$

Take this equation mod 4.

$$1\equiv x_1^2 + x_2^2 + x_3^2 \pmod 4.$$

It is easy to see that the only parities of $x_1,x_2,x_3$ that work
are for two of them to be even and one of them to be odd.

\smallskip

\noindent
3) Let $n\in\N$. Note that $4n+1$ satisfies the premise of Part 2. 
By Part 2 there exists $x_1,x_2,x_3\in\Z$ such that 

$$4n+1 = (2x_1)^2 + (2x_2)^2 + (2x_3+1)^2$$

$$4n+1 = 4x_1^2 + 4x_2^2 + 4x_3^2 + 4x_3 + 1$$

$$n = x_1^2 + x_2^2 + x_3^2+ x_3.$$
\end{proof}

\begin{samepage}

\begin{theorem}\label{th:HtenZD}~
\begin{enumerate}
%\item
%If $\HtenZ(2d,4n)=\D$, then $\HtenN(d,n)=\D$.
%\item
%If $\HtenN(d,n)=\U$, then $\HtenZ(2d,4n)=\U$.
%This is the contrapositive of part 1.
\item
If $\HtenZ(2d,3n)=\D$, then $\HtenN(d,n)=\D$.
% 1
\item
If $\HtenN(d,n)=\U$, then $\HtenZ(2d,3n)=\U$.
This is the contrapositive of Part 1.
% 2
\item
If $\HtenZ(f(d,n),2n+2)=\D$, then $\HtenN(d,n)=\D$
where
$$f(d,n)=\max\{2d,(2n+3)2^n\}.$$
% 3
\item
If $\HtenN(d,n)=\U$, then $\HtenZ(f(d,n),2n+2)=\U$.
This is the contrapositive of Part 3.
% 4
\end{enumerate}
\end{theorem}

\end{samepage}

\begin{proof}~

%\noindent
%1) Assume $\HtenZ(2d,4n)=\D$.  We show that $\HtenN(d,n)=\D$.
%
%Let $p\in \Z[x_1,\ldots,x_n]$. We want to know if there is a solution in $\N$.
%
%Let $q$ be the polynomial of degree $2d$ with $4n$ variables that you get if you replace
%each $x_i$ with
%$x_{i1}^2+x_{i2}^2+x_{i3}^2+x_{i4}^2$
%where
%$x_{i1},x_{i2},x_{i3},x_{i4}$
%are 4 new variables.
%By Lemma~\ref{le:three}.1 we have:
%
%$p$ has a solution in $\N$ iff $q$ has a solution in $\Z$.
%
%Use that $\HtenZ(2d,4n)=\D$ to determine if $q$ has a solution. Hence
%$\HtenN(d,n)=\D$.
%
%\bigskip
%
\noindent
1) Let $p\in \Z[x_1,\ldots,x_n]$. We want to know if there is a solution in $\N$.

Let $q$ be the polynomial of degree $2d$ with $3n$ variables that you get if you replace
each $x_i$ with
$x_{i1}^2+x_{i2}^2+x_{i3}^2+x_{i3}$
where
$x_{i1},x_{i2},x_{i3}$
are 3 new variables.
By Lemma~\ref{le:three}.3 we have:

$p$ has a solution in $\N$ iff $q$ has a solution in $\Z$.

Use that $\HtenZ(2d,3n)=\D$ to determine if $q$ has a solution. Hence $\HtenN(d,n)=\D$.

\bigskip

\noindent
3) This was proven by Sun~\cite{zhi}. 

\end{proof}

\begin{theorem}\label{th:HtenND}~
\begin{enumerate}
\item
If $\HtenN(d,n)=\D$ then $\HtenZ(d,n)=\D$.
\item
If $\HtenZ(d,n)=\U$ then $\HtenN(d,n)=\U$.
This is the contrapositive of part 1.
\end{enumerate}
\end{theorem}

\begin{proof}
Let $p\in \Z[x_1,\ldots,x_n]$. We want to know if there is a solution in $\Z$.
For each $\vec b = (b_1,\ldots,b_n) \in \bits n$ let
$q_{\vec b}(x_1,\ldots,x_n)$ be formed as follows: for every $i$ where $b_i=1$, replace $x_i$ with $-x_i$.
It is easy to see that

$p$ has a solution in $\Z$ iff

$(\exists \vec b)[q_{\vec b}\hbox{ has a solution in $\N$ }].$

The result follows.
\end{proof}

\section{What Happens for Fixed $d,n$?}\label{se:dn} 

\subsection{When is $\HtenN(d,n)=\U$? $\HtenZ(d,n)=\U$?}\label{se:H10U}

In 1980 Jones~\cite{Jones-1980} announced 16 pairs $(d,n)$ for which $\HtenN(d,n)=\U$. 
In 1982 Jones~\cite{Jones-1982} provided proofs for 13 of these pairs
(12 in Theorem 4 and 1 in Section 3). 
I emailed Jones about the other three and he emailed back the following:
\begin{itemize}
\item
Those with $d<2668$ have proofs similar to the $(4,58)$ case. 
This was carried out by Dr.\ Hideo Wada. (No reference is given.)
\item 
The pair with a very large value of $d$ can be obtained
using many relation-combining theorems, like the one at the end of
the 1982 paper, which allow one to define two squares with one unknown. 
\end{itemize}

In the theorem below we present all 16 statements from the Jones-1980 paper
along with a result by Sun~\cite{Sun-2020} from 2020. 
We note (1) which three do not have proofs in Jones-1982 (though based on
Jones's email we are sure the results are true), and (2) the result of Sun. 
We state the results of the form $\HtenN(d,n)=\U$ and then 
apply Theorem~\ref{th:HtenZD} to obtain results of the form $\HtenZ(d',n')=\U$
(except for Sun's result which is already about $\HtenZ$). 

The proofs involve very clever use of elementary number theory to 
get the degrees and number-of-variables reduced. 

In some of the results there are absurdly large numbers like 
$4.6\times 10^{44}$. These are probably upper bounds that might be able to be lowered
with a careful examination of the proofs. 
These large numbers only occur as $d$ since the main concern was to get the
number of variables down. 

\begin{theorem}\label{th:h10N}~
\begin{enumerate}
\item
$\HtenN(4,58)=\U$ hence
$\HtenZ(8,174)=\U$.

\item
$\HtenN(8,38)=\U$ hence
$\HtenZ(16,114)=\U$.

\item
$\HtenN(12,32)=\U$ hence
$\HtenZ(24,96)=\U$.

\item
$\HtenN(16,29)=\U$ hence
$\HtenZ(32,87)=\U$.
(Not proven in Jones-1982.)

\item
$\HtenN(20,28)=\U$ hence
$\HtenZ(40,84)=\U$.

\item
$\HtenN(24,26)=\U$ hence
$\HtenZ(48,78)=\U$.

\item
$\HtenN(28,25)=\U$ hence
$\HtenZ(56,75)=\U$.

\item
$\HtenN(36,24)=\U$ hence
$\HtenZ(72,72)=\U$.
(Not proven in Jones-1982.)

\item
$\HtenN(96,21)=\U$ hence
$\HtenZ(192,63)=\U$.

\item
$\HtenN(2668,19)=\U$ hence
$\HtenZ(5336,57)=\U$.
%$\HtenZ(41 \times 2^{19},40)=\U$.

\item
$\HtenN(200000,14)=\U$ hence

$\HtenZ(400000,42)=\U$ and 
$\HtenZ(31 \times 2^{14},30)=\U$.

\item
$\HtenN(6.6\times 10^{43},13)=\U$ hence
$\HtenZ(13.2 \times 10^{43},28)=\U$.
(Not proven in Jones-1982.)

\item
$\HtenN(1.3\times 10^{44},12)=\U$ hence
$\HtenZ(2.6 \times 10^{44},36)=\U$.

\item
$\HtenN(4.6\times 10^{44},11)=\U$ hence
$\HtenZ(9.2 \times 10^{44},24)=\U$.

\item
$\HtenN(8.6\times 10^{44},10)=\U$ hence
$\HtenZ(17.2 \times 10^{44},22)=\U$.

\item
$\HtenN(1.6\times 10^{45},9)=\U$ hence
$\HtenZ(3.2 \times 10^{45},20)=\U$.
(Jones' 1982 paper presents the proof of this result and credits it to 
Matijasevi\v{c}.)

\item 
$\HtenZ(d,11)=\U$ for some $d$. The number $d$ is not stated. 
(This is due to Sun~\cite{Sun-2020}.)

\end{enumerate}
\end{theorem}

\subsection{When is $\HtenZ(d,n)=\D$? $\HtenN(d,n)=\D$?}\label{se:H10D}

We will need a brief discussion of the following problem which is attributed to Frobenius.

\noindent
{\it Given a set of relatively prime positive integers $\vec a = (a_1,\ldots,a_n)$ find
the set 

$$\FROB(\vec a)=\biggl \{ \sum_{i=1}^n a_ix_i \colon x_1,\ldots,x_n\in\N \biggr \}.$$
}

\bigskip

It is known that $\FROB(\vec a)$ is always cofinite. 
We will need to look at the case where the $a_1,\ldots,a_n$ may have a
gcd of $d\ne 1$. In this case, $\FROB(\vec a$) is always a cofinite subset of $d\N$. 

The $n=2$ case was solved by James Joseph Sylvester in 1884:

\begin{lemma}\label{le:syl}
Let $a_1,a_2 \in\N$. Let $d=\gcd(a_1,a_2)$. 
There exists a finite set $F\subseteq d\N$ such that 

$$\FROB(a_1,a_2) = F \cup \{ dx \colon x\ge a_1a_2-a_1-a_2+1 \}$$

\noindent
and $(a_1a_2-a_1-a_2)d\notin \FROB(a_1,a_2)$. 
\end{lemma}

For the general case there is no neat formula; however, finding $\FROB(\vec a)$
is decidable. There has been much work on this problem. 
Beihoffer et al.~\cite{BHNW-2005} gives a fast algorithm and 
many prior references to other algorithms. We state the relevant lemma. 

\begin{lemma}\label{le:frob}
Let $a_1,\ldots,a_n \in\N$. Let $d=\gcd(a_1,\ldots,a_n)$. 
\begin{enumerate}
\item
There exists  finite $F$ and an $M\in\N$ such that 

$$\FROB(\vec a) = F \cup \{ dx \colon x\ge M\}$$

and 

$(M-1)d\notin \FROB(\vec a)$. 

\item 
There is an algorithm that will, given $a_1,\ldots,a_n$, find $F$ and $M$. 
\end{enumerate}
\end{lemma}

And now for the main theorem of this section.

\begin{theorem}\label{th:h10Z}~
\begin{enumerate}

\item
For all $d$, $\HtenZ(d,1)=\D$ and $\HtenN(d,1)=\D$. 
There is an algorithm that finds all of the integer roots
(which may be the empty set). 

\item
For all $n$, $\HtenZ(1,n)=\D$. 

\item 
For all $n$, $\HtenN(1,n)=\D$. 

\item
$\HtenZ(2,2)=\D$. 

\item 
$\HtenN(2,2)=\D$. 

\item
For all $n$, $\HtenZ(2,n)=\D$.

\item
For all $n$, $\HtenN(2,n)=\D$. 

\end{enumerate}
\end{theorem}

\begin{proof}~

\noindent
1) These are both  easy consequences of the rational root theorem: 
If $a_dx^d + \cdots a_1x + a_0\in\Z[x]$ has a rational root $\frac{p}{q}$ then $p$ divides $a_0$ 
and $q$ divides $a_n$.

The above algorithm does not find the roots. One can modify the algorithm so that it does
find the roots; however, that would be a slow algorithm. 
Cucker et al.~\cite{CKS-1999} gave a polynomial time algorithm for finding the
set of integer roots. 

\smallskip

\noindent
2) Given $\sum_{i=1}^n a_ix_i = b$ where $a_1,\ldots,a_n,b\in\Z$, we need to determine if there is 
a solution in $\Z$. 

First find $d=\gcd(a_1,\ldots,a_n)$. 
If $d$ does not divide $b$ then there are no solutions in $\Z$.
If $d$ does divide $b$ then there is a solution in $\Z$:
Let $x_1',\ldots,x_n'$ be such that $\sum_{i=1}^n a_ix_i'=d$
and let $x_i= \frac{bx_i'}{d}$. 

\smallskip

\noindent
3) We can phrase any problem we need to solve as follows:
Let $a_1,\ldots,a_n\in\N$ and $b_1,\ldots,b_m,b\in\N$. Is there a solution in $\N$ of 

$$\sum_{i=1}^n a_ix_i = b+ \sum_{i=1}^m b_iy_i\hbox{?}$$

Let $d_a=\gcd(a_1,\ldots,a_n)$ and $d_b=\gcd(b_1,\ldots,b_m)$. 

By Lemma~\ref{le:frob}:
\begin{itemize}
\item
There is an algorithm that will find finite set $F_a$ and an $M_a\in\N$ such that 

$$
\biggl 
\{ 
\sum_{i=1}^n a_ix_i \colon x_1,\ldots,x_n\in\N
\biggr
\}
= F_a \cup \{xd \colon x\ge M_a
\}
$$

\item 
There is an algorithm that will find finite set $F_b$ and an $M_b\in\N$ such that 

$$
\biggl 
\{ 
b+\sum_{i=1}^n b_ix_i \colon x_1,\ldots,x_n\in\N
\biggr \}= 
F_b \cup \{b+xd \colon x\ge M_b
\}
$$
\end{itemize}

Once we have $F_a,M_a,F_b,M_b$ it is easy to determine if 
$
\{ 
F_a \cup \{xd \colon x\ge M_a
\}
$
and 
$
\{ 
F_b \cup \{b+xd \colon x\ge M_b
\}
$
intersect. If so, then there is a solution to the original equation, and if not,
then there is not. 

\smallskip
4) Gauss~\cite{Gauss-1986} (27, Art, 216-221)  proved this.
For a more modern approach, 
Lagarias~\cite{Lagarias-1979} (Theorem 1.2.iii)  
showed that if $p(x,y)\in \Z[x,y]$ of degree 2 
has a solution then there is a short proof for this fact (short means
of length bounded by a polynomial in the size of the coefficients).
Formally he showed that the
following set is in $\NP$.

$$\{ (a,b,c,d,e,f) \in \Z^6 \colon (\exists x,y\in\Z)[ax^2 + bxy + cx^2 + dx + ey + f = 0] \}.$$

(There is a solver on the web here: 

\centerline{\url{https://www.alpertron.com.ar/QUAD.HTM} )}

\smallskip
5) Gauss's method to determine if  $f(x,y)\in \Z[x,y]$, of degree 2,
has a solution in $\Z$ finds all of the solutions in a nice form.
From this form one can determine if there are any solutions in $\N$. 

\smallskip
6) For all $n$, $\HtenZ(2,n)=\D$.
This is a sophisticated theorem due to Siegel~\cite{siegelhten}. 
See also a simpler (though still difficult) proof by
Grunewald and Segal~\cite{GS-1981}.

\smallskip
7) For all $n$, $\HtenN(2,n)=\D$.
This is a sophisticated theorem due to Grunewald and Segal~\cite{gsHten}.
Their proof uses the Hasse-Minkowski Theorem (see Page 32 of Grunewald-Segal).
\end{proof}

\subsection{The Curious Case of $\HtenZ(3,2)$}\label{se:32}

We give evidence that $\HtenZ(3,2)=\D$; however, this is still open.

\begin{definition}
An element of $\Q[x_1,\ldots,x_n]$ is {\it absolutely irreducible} 
if it is irreducible over $\C$. For example, 

$x^2+y^2-1$ is absolutely irreducible, but 

$x^2+y^2=(x+iy)(x-iy)$ is not. 
\end{definition}

A combination of results by Baker and Cohen~\cite{BC-1970}, 
Poulakis~\cite{Poul-1993}, and
Poulakis~\cite{Poul-2002} imply the following theorem:

\begin{theorem}
There is an algorithm which, given any absolutely irreducible polynomial
$P(x,y)\in\Z[x,y]$ of degree 3, determines all integer solutions of the
equation $P(x,y)=0$.
(See 
Poulakis~\cite{Poul-2002} for a more precise definition of 
``determines all integer solutions'' in the case that there are an infinite 
number of them.)
\end{theorem} 

The original algorithm (from Baker and Coates) is not practical; however, 
Peth\H{o} et al.~\cite{PZGH-1999} and Stroker-Tzankis~\cite{ST-2003} 
have practical algorithms. There is also an algorithm for solving a large class
of cubic equations implemented in SageMath. 

So why isn't $\HtenZ(3,2)=\D$? Because the case where $P(x,y)$ has degree 3 but 
is not absolutely irreducible is still open.

\section{Particular Equations}\label{se:part} 

\subsection{If the Variables Are Separated$\ldots$}\label{se:sep} 

Ibarra and Dang~\cite{ID-2006} proved the following. 

\begin{definition}
$P(z_1,\ldots,z_n)$ is a {\it Presburger Relation} if it can be expressed with 
$\Z$, $=,+,<$, and the usual logical symbols. For example 

$(z_1 + z_2 < z_3 + 12) \wedge (z_1 + z_4 = 17)$ is a Presburger formula, but

$z_1z_2=13$ is not. 
\end{definition}

\begin{theorem}
The following is decidable:

\noindent
{\bf Instance} 

\noindent
(1) For $1\le i\le k$, polynomial $p_i(y)\in \Z[y]$, and linear functions 
$F_i(\vec x),G_i(\vec x)\in \Z[x_1,\ldots,x_n]$, and 
(2) a Presburger relation $R(z_1,\ldots,z_k)$.

\noindent
{\bf Question} Does there exist $y,\vec x$ such that 

$$R(p_1(y)F_1(\vec x)+G_1(\vec x),\ldots,p_k(y)F_k(\vec x)+G_k(\vec x)\,)$$

\noindent 
holds?
\end{theorem}

\subsection{The Curious Case of $x^3+y^3+z^3=k$}\label{se:xyzk}

Rather than looking at $\HtenZ(d,n)$ let's focus on one equation
that has gotten a lot of attention: 

$$x^3+y^3 + z^3 = k.$$

It is easy to show that, for $k\equiv 4,5 \pmod 9$, there is no solution in $\Z$. 
What about for $k\not\equiv 4,5 \pmod 9$? 
\begin{enumerate}
\item
Heath-Brown~\cite{HB-1992} conjectured that there are an infinite number of
$k\not\equiv 4,5 \pmod 9$ for which there is a solution in $\Z$. 
Others think that, for all 
$k\not\equiv 4,5 \pmod 9$, $x^3+y^3+z^k=k$  has a solution in $\Z$. 
\item 
Elkies~\cite{Elkies-2000} devised an efficient algorithm to find solutions to

$x^3+y^3+z^3=k$ if there is a bound on $x,y,z$. 
\item
Elsehans and Jahnel~\cite{EJ-2009} modified and implemented Elkies algorithm and 
determined the following: 
The only $k\le 1000$, $k\not\equiv 4,5 \mod 9$, where
they did not find a solution were

33, 42, 74, 114, 165, 390, 579, 627, 633, 732, 795, 906, 921, and 975.

Their work, and the work of all the items below, required hard mathematics,
clever computer science, and massive computer time. 
\item 
Huisman~\cite{Huisman-2016} found a solution for $k=74$. For many other values of $k$ where there
were solutions, Huisman found additional solutions. 
\item 
Booker~\cite{Booker-2019} found a solution for $k=33$. 
\item 
Booker found solutions for $k=42$ and $k=795$. These have not been formally published yet; 
however, the $x,y,z$ can be found on the Wikipedia site:

\url{https://en.wikipedia.org/wiki/Sums_of_three_cubes}

\item
As of April 2021 (when this article was written) 
the only $k\le 1000$, $k\not\equiv 4,5 \mod 9$,  where
no solution is known are: 

114, 165, 390, 579, 627, 633, 732, 906, 921, and 975.
\end{enumerate} 

Consider the function that, on input $k$, determines if $x^3+y^3+z^3=k$ has
a solution in $\Z$. Is this function computable? 
\begin{enumerate}
\item 
I suspect the function is computable.
Why? What would a proof that this function is not computable look like? It would have
to code a Turing machine computation into a very restricted equation. 
This seems unlikely to me. 
Note also that it may be the case the equation has a solution for every
$k\not\equiv 4,5 \pmod 9$, in which case the decision problem is not just decidable---it's regular!
\item 
Daniel Varga has suggested there may be a proof that does not go through
Turing machines. Perhaps some other undecidable problem? Also, there may be
new techniques we just have not thought of yet. 
\end{enumerate}

\section{Discussion}\label{se:disc}

If I was to draw the grid for $\HtenN$ or $\HtenZ$ mentioned in the introduction
there would be a large space of problems that are open. 
We give an example of a part of that space.

Recall that 
$\HtenZ(d,1)=\D$,  $(\forall n)[\HtenZ(2,n)=\D]$, and $\HtenZ(8,174)=\U$. 
The following are unknown:
\begin{enumerate}
\item 
$\HtenZ(3,2), \HtenZ(3,3), \HtenZ(3,4),\ldots$ .
\item 
$\HtenZ(4,2), \HtenZ(4,3), \HtenZ(4,4),\ldots$ .
\item 
$\HtenZ(5,2), \HtenZ(5,3), \HtenZ(5,4),\ldots$ .
\item 
$\HtenZ(6,2), \HtenZ(6,3), \HtenZ(6,4),\ldots$ .
\item 
$\HtenZ(7,2), \HtenZ(7,3), \HtenZ(8,4),\ldots$ .
\item 
$\HtenZ(8,2), \HtenZ(8,3), \HtenZ(8,4),\ldots, \HtenZ(8,173)$. 
\end{enumerate}

The situation is worse than it looks. From the discussion in Section~\ref{se:xyzk}
we know that the status of the following function is unknown: Given $k$, determine if 
$x^3+y^3+z^3=k$ has a solution in $\Z$.

What is the smallest $n$ such that, for some $d$, $\HtenZ(d,n)=\U$? We present an
informed opinion by paraphrasing and combining two passages from Sun~\cite[pages 209 and 211]{zhi}: 

\begin{enumerate}
\item
Matijasevi\v{c} and Robinson~\cite{MR-1975} showed there is a $d$ such that $\HtenN(d,13)=\U$. 
\item 
Matijasevi\v{c} showed there is a $d$ such that $\HtenN(d,9)=\U$. 
By Theorem~\ref{th:HtenZD} we have that there is a $d'$ with $\HtenZ(d',20)=\U$. 
\item 
Baker~\cite{Baker-1968} showed the following is decidable: Given $p\in \Z[x,y]$, $p$ homogenous,
does it have a solution in $\Z$?
This does \emph{not} show that
 $$(\forall d)[\HtenZ(d,2)=\D]$$ 
\noindent
but it points in that direction. 
\item 
(Direct quote from page 209.) {\it In fact, A.~Baker, 
Matijasevi\v{c} 
and Robinson even conjectured that $\exists^3$ is undecidable over $\N$.}
In our notation, there exists $d$ such that $\HtenN(d,3)=\U$.
\end{enumerate}

Suggestions:
\begin{enumerate}
\item
Since a grid for $\HtenN(d,n)$ or $\HtenZ(d,n)$ is somewhat cumbersome there should be
a website of results. 
\item 
That website should also include classes of equations such as $x^3+y^3+z^3=k$
and what is known about them. 
\item 
Work on showing $\HtenN(d,n)=\U$ or $\HtenZ(d,n)=\U$ seems to have stalled. Perhaps
the problems left are too hard. Perhaps the problems left could be resolved but it
would be very messy. Perhaps computer-work could help (see next point). 
Perhaps deeper number theory is needed (current results seem to use clever but somewhat elementary 
number theory). 
Perhaps the problems left are decidable. 
In any case, there should be an effort in this direction. 
\item
There has been some work on getting Universal Turing machines down to a
very small number of states and alphabet size.  See, for example, the
work of 
Aaronson~\cite{Aaronson-2020}, 
Michel~\cite{Michel-2013}, 
Yedidia and Aaronson~\cite{YA-2016}, 
See also the following blog post on this site: \url{https://vzn1.wordpress.com}
that you get by clicking on MENU and looking for 
\emph{Undecidability: The Ultimate Challenge}. 

There has even been some computer work done in writing compilers for
these machines. It is plausible that by starting from these rather
small machines, smaller polynomials may suffice to simulate them. 
\end{enumerate}

\section{Variants that Use Fewer Variables}\label{se:variants}

Hilbert's 10th problem, and the restrictions on it in this article,
are about the solvability of the following problem: Given
$p(x_1,\ldots,x_n)\in\Z[x_1,\ldots,x_n]$, is the following true over $\Z$: 

$$(\exists x_1)\cdots(\exists x_n)[p(x_1,\ldots,x_n)=0].$$

(Undecidability results were usually about truth over $\N$.) 

There has been much work in getting the number of variables needed
for an undecidability result to be small. 
As we saw in Theorem~\ref{th:h10N}, $\HtenN(1.6\times 10^{45},9)=\U$. 
As of April 2021 (when this was written) 9 is the lowest $n$ such that
there is known to be a $d$ with $\HtenN(d,n)=\U$. 
The result of 9 was proven by Matijasevi\v{c} in the early 1980's
(it appears in Jones~\cite{Jones-1982} and credited to 
Matijasevi\v{c}
). Hence the 9 has not been improved
in 29 years. I doubt it will be improved between writing this paper
and the appearance of this paper.
As we saw in Theorem~\ref{th:h10N}, there is a $d$ such that $\HtenZ(d,11)=\U$. 
This was proven in 2020 so it is plausible to be improved in the near future. 

We explore some variants of H10 where the number of variables needed
is smaller than 9 (for $\N$) and 11 (for $\Z$). 

\subsection{Different Quantifier Prefixes} 
Let $Q_1\cdots Q_n$ be a string of quantifiers. 
Consider the following problem. Given $p(x_1,\ldots,x_n)$ is 

$$(Q_1 x_1)\cdots(Q_n x_n)[p(x_1,\ldots,x_n)=0]$$

\noindent
true over $\Z$? Over $\N$?

\begin{notation}~
\begin{enumerate}
\item 
Let $Q_1\cdots Q_n$ be a string of quantifiers. 
{\it $Q_1\cdots Q_n$ is undecidable over $\N$} if
the above problem is undecidable over $\N$.
Similar for $\Z$. 
\item
A quantifier is \emph{bound} if there is an explicitly upper and lower bound on it
which is a polynomial in the prior variables.
\end{enumerate} 
\end{notation}

Recall from Theorem~\ref{th:h10N} that $\exists^9$ is undecidable over $\N$
and that this is the best known. 

Matijasevi\v{c}~\cite{YuriM-1972} showed that $\exists \forall \exists^2$,
with $\forall$ bounded, over $\N$, is undecidable. 
From this result one can obtain 
undecidability with polynomials of four variables. 
This is much better than nine. 
See Sun~\cite{Sun-2021} for more of history, references, and results about 
quantifier prefixes and undecidability over $\N$. 

Recall from Theorem~\ref{th:h10N} that $\exists^{11}$ is undecidable over $\Z$
and that this is the best known. 

Sun~\cite{Sun-2021} proved the following.
\begin{enumerate}
\item 
These are undecidable over $\Z$:
$\forall \exists^7$, 
$\forall^2 \exists^4$,
$\exists\forall\exists^4$,
$\exists\forall^2\exists^3$,
$\exists^2\forall\exists^3$, 
$\forall\exists\forall\exists^3$,
$\forall\exists^2\forall^2\exists^2$,
$\forall^2\exists\forall^2\exists^2$,
$\forall\exists\forall^3\exists^2$, 
$\exists^2\forall^3\exists^2$,
$\exists\forall\exists\forall^2\exists^2$,
$\exists\forall^6\exists^2$. 
Note that the shortest prefixes only use 6 variables which is much better than 11. 
\item 
These are undecidable if the $\forall$ are bounded:
$\exists\forall\exists^3$,
$\exists\forall^2\exists^3$,
$\exists^2\forall^2\exists^3$,
$\exists^2\forall^2\exists^2$,
$\exists^2\forall\exists\forall\exists^2$,
$\exists\forall^5\exists^2$. 
Note that the shortest prefixes only use 5 variables which is much better than 11. 
\end{enumerate}

\subsection{Sets of Polynomials} 

Matijasevi\v{c}
and Robinson~\cite{MR-1996} 
(see also 
Matijasevi\v{c}~\cite{YuriM-1972})
prove the following
(All quantifiers are over $\N$). 
Let $A$ be an r.e.\ set. 
\begin{enumerate}
\item 
There exist $3n$ polynomials 

$\{P_i(x_1,x_2,x_3)\}_{i=1}^n$,
$\{Q_i(x_1,x_2,x_3)\}_{i=1}^n$,
$\{R_i(x_1,x_2,x_3)\}_{i=1}^n$ such that 

$$a\in A \hbox{ iff }$$ 

$$(\exists b,c) \bigwedge_{i=1}^n (\exists d)[ P_i(a,b,c) < Q_i(a,b,c)\times d < R_i(a,b,c)].$$

From this result one can obtain a problem with polynomials in 3 variables
that is undecidable.

\item 
There exist polynomials 

$P(x_1,x_2,x_3)$ and 
$Q(x_1,x_2,x_3,x_4)$ such that 

$$a\in A \hbox{ iff }$$

$$(\exists b,c)(\forall f)[(f\le P(a,b,c)) \implies (Q(a,b,c,f)>0)]$$

From this result one can obtain a problem with polynomials in 3 variables
that is undecidable.

\end{enumerate} 

\section{What Would Hilbert Do?}\label{se:disc2}

\begin{definition}
$\HtenQ(d,n)$ is the problem where  the degree is $\le d$, the number of
variables is $\le n$, and we seek a solution in $\Q$.
\end{definition} 

Matijasevi\v{c}~\cite{Mat-uuuu} (Page 18) gives good reasons why Hilbert might
have actually wanted to solve $\HtenQ$. Hilbert stated 
the tenth problem as $\HtenZ$; however, if $\HtenZ$ is solvable then
$\HtenQ$ is solvable. He might have thought that the best way to solve
$\HtenQ$ is to solve $\HtenZ$. 

What is the status of $\HtenQ$ now? It is an open question to determine if
$\HtenQ$ is decidable. Hence the problem  Hilbert plausibly intended to ask is
still open and may yet lead to number theory of interest, which was his intent. 

\section{Acknowledgement}

We thank 
Blogger vzn, 
Timothy Chow, 
Thomas Erlebach,
Stephen Fenner, 
Lance Fortnow, 
Brogdan Grechuk, 
Nathan Hayes,
James Jones, 
Emily Kaplitz, 
Chris Lastowski,
David Marcus, 
Yuri Matijasevi\v{c},
Andras Salamon, 
Dan Segal, 
Yuang Shen,
Joshua Twitty,
Larry Washington,
Daniel Varga, 
Zan Xu,
for helpful discussions. 

We are particularly grateful to the following people. 
\begin{enumerate}
\item 
Timothy Chow for his comments in Section~\ref{se:h10} and help with the
discussion in Section~\ref{se:xyzk} of $x^3+y^3+z^3=k$. 
\item 
Brogdan Grechuk for telling us about the material that is now in Section~\ref{se:32}.
\item
James Jones for discussion of Theorem~\ref{th:h10N}.
\item
Yuri Matijasevi\v{c}
for pointing us to many results of which we were unaware. 
\end{enumerate}

\newcommand{\etalchar}[1]{$^{#1}$}

 % \bibliographystyle{alpha}
 % \bibliography{bibfile}
\end{document}